\documentclass[11pt]{article}
\usepackage{amsmath}
\usepackage{amssymb}
\usepackage{amstext}
\usepackage{amsthm}
\usepackage{mathtools}
\usepackage{url}

\newcounter{defcounter}
\newtheorem{theorem}{Theorem}

\newtheorem{lemma}[theorem]{Lemma}

\newtheorem{conjecture}[theorem]{Conjecture}

\newtheorem{algorithm}[theorem]{Algorithm}

\numberwithin{equation}{section}

\def\R{{\mathbb R}}

\begin{document}

\title{A Fast Self-correcting $\pi$ Algorithm}
\author{{\Large Tsz-Wo Sze} \\
{\small szetszwo@apache.org} \\
{\small \url{https://home.apache.org/~szetszwo/}}
}
\maketitle
\begin{abstract}
We have rediscovered a simple algorithm to compute the mathematical constant
\[
\pi=3.14159265\cdots.
\]
The algorithm had been known for a long time
but it might not be recognized as a fast, practical algorithm.
The time complexity of it can be proved to be
\[
O(M(n)\log^2 n)
\]
bit operations
for computing $\pi$ with error $O(2^{-n})$,
where $M(n)$ is the time complexity to multiply two $n$-bit integers.
We conjecture that the algorithm actually runs in
\[
O(M(n)\log n).
\]
The algorithm is \emph{self-correcting} in the sense that,
given an approximated value of $\pi$ as an input,
it can compute a more accurate approximation of $\pi$ with cubic convergence.
\end{abstract}
\section{Introduction}
The computation of the mathematical constant $\pi$
has drawn a great attention from mathematicians and computer scientists
over the centuries \cite{Berggren2004, Knuth1997}.
The known asymptotically fastest algorithms for computing $\pi$ run in
\[
O(M(n)\log n)
\]
bit operations with error $O(2^{-n})$,
where $M(n)$ is the time complexity to multiply two $n$-bit integers.
The AGM algorithms \cite{Brent1976pi, Salamin1976, Borwein1987} are the only examples.
If the recent result in \cite{harvey2019} is correct,
\[
M(n) = O(n\log n).
\]
Then, the known asymptotically fastest algorithms run in
\[
O(n\log^2 n).
\]
The Chudnovsky algorithm \cite{Chudnovsky1989},
which runs in
\[
O(M(n)\log^2 n),
\]
is a popular implementation choice.
The computer program, \emph{y-cruncher}, implemented the Chudnovsky algorithm
has been used to compute $\pi$ to 31.4 trillion digits \cite{Yee2019, Iwao2019}.

In this paper,
we revisit a simple algorithm, Algorithm~\ref{alg:pi}, to compute $\pi$.
The algorithm had been known to Salamin \cite{Beeler1972}
but it might not be recognized as a fast, practical algorithm.
The time complexity of it can be proved to be
\[
O(M(n)\log^2 n).
\]
It is \emph{self-correcting} in the sense that,
given an approximated value of $\pi$ as an input,
it can compute a more accurate approximation of $\pi$ with cubic convergence.
If there is an $O(M(n)\log n)$ algorithm to compute $\sin x$ without requiring $\pi$
for any $n$-bit number $x$ with $|x|<U$, a fixed upper bound,
then the algorithm runs in $O(M(n)\log n)$.
We have the following conjecture.
\begin{conjecture}\label{conj:O(M(n)log n)}
Algorithm~\ref{alg:pi} runs in $O(M(n)\log n)$ bit operations.
\end{conjecture}

Similar to the Chudnovsky algorithm, Algorithm~\ref{alg:pi} uses binary splitting.
According to \cite{mca2010, Yee2019faq},
binary splitting has advantages over AGM including:
\begin{enumerate}
\item the implicit constants for binary splitting are smaller than the ones for AGM;
\item binary splitting can be speeded up by simultaneously summing up many terms at once
but it is difficult to speed up AGM;
\item AGM has very poor memory locality.
\end{enumerate}
Moreover, the AGM iteration is not self-correcting so that full precision is required throughout.
In contrast, the intermediate results can be truncated in Algorithm~\ref{alg:pi}.
For example, suppose the current step has computed $\pi$ in $m$ decimal places
and the next target precision is $n$ decimal places for some $n\leq 3m$.
Then the current result can be truncated to roughly $n/3$ decimal places.
Thus, Algorithm~\ref{alg:pi} potentially runs faster than the AGM algorithms in practice.

The algorithm is presented in the next section.
We discuss the $\pi$ verification problem in Section 3.
Finally, we show a family of sequences converging to $\pi\pmod{2\pi}$ in Section 4.
\section{The Computational Problem}
Let $\alpha$ be an approximated value of $\pi$ and
\begin{equation}\label{eqn:delta=pi-alpha}
\delta =\pi - \alpha
\end{equation}
be the error with
\begin{equation}\label{eqn:delta<epsilon}
|\delta| < \epsilon<1
\end{equation}
for some fixed $\epsilon$.
By the Taylor series
\begin{equation*}\label{eqn:sin-taylor}
\sin(x)
= x-\frac{x^3}{3!} + \frac{x^5}{5!} - \frac{x^7}{7!} + \frac{x^9}{9!} - \frac{x^{11}}{11!} + \dots,
\end{equation*}
it is easy to see that
\begin{equation}\label{eqn:delta-sin-alpha}
|\delta - \sin\delta| < \frac{|\delta|^3}{6} < \frac{\epsilon^3}{6}.
\end{equation}
Note that
\begin{equation}\label{eqn:sin-delta=sin-alpha}
\sin\delta = \sin\alpha.
\end{equation}
Finally, we obtain a better approximated value of $\pi$
\begin{equation}\label{eqn:alpha'}
\alpha' = \alpha + \sin\alpha
\end{equation}
such that the error
\begin{equation}\label{eqn:pi-alpha'}
|\pi - \alpha'| < \frac{\epsilon^3}{6}
\end{equation}
becomes cubic
by (\ref{eqn:delta=pi-alpha}), (\ref{eqn:delta<epsilon}),
(\ref{eqn:delta-sin-alpha}), (\ref{eqn:sin-delta=sin-alpha}) and (\ref{eqn:alpha'}).
Note that $\alpha < \pi$ implies $\alpha<\alpha' < \pi$
since the inequalities (\ref{eqn:delta<epsilon}), (\ref{eqn:delta-sin-alpha}) and (\ref{eqn:pi-alpha'})
still hold after dropped all absolute value functions.
Similarly, $\alpha > \pi$ implies $\alpha>\alpha' > \pi$.
We have proved the following theorem,
which is known to Salamin \cite{Beeler1972}.

\begin{theorem}[Cubic Convergence]\label{thm:cubic}
Let $\alpha$ be an approximated value of $\pi$ such that $|\pi-\alpha|<\epsilon<1$,
Then,
$|\pi-\alpha'|<\epsilon^3/6$,
where $\alpha' = \alpha + \sin\alpha$.
\end{theorem}

We present Algorithm~\ref{alg:pi} below and then prove its time complexity.

\begin{algorithm}[Self-correcting $\pi$ Computation]\label{alg:pi}
The input is a positive integer $n$.
This algorithm returns $\alpha$ such that $\pi-\alpha< 2^{-n}$.
\end{algorithm}
\begin{enumerate}
\item[I.]
Let $\alpha_0=3$ and $m=\lceil\log_3n\rceil$.
\item[II.]
For $k=1,2,\ldots,m$,
use $(3^k)$-bit precision to compute
\[
\alpha_k=\alpha_{k-1} + \sin\alpha_{k-1}.
\]
\item[III.] Return $\alpha_m$.
\end{enumerate}

\begin{theorem}\label{thm:pi-time}
Algorithm~\ref{alg:pi} runs in $O(M(n)\log^2 n)$ bit operations.
\end{theorem}
\begin{proof}
The main ingredient of Algorithm~\ref{alg:pi} is to compute $\sin\alpha_k$.
For $x$ a $\ell$-bit number with $0<x<1/2$,
both $\sin x$ and $\cos x$ can be computed
using the $O(M(\ell)\log^2 \ell)$ binary splitting SinCos algorithm \cite{mca2010}.
Therefore,
$\sin(\alpha_k/8)$ and $\cos(\alpha_k/8)$ can be computed in $O(M(3^k)k^2)$.
Then use the doubling formulae
\begin{align*}
\sin(2x) &= 2\sin x\cos x, \\
\cos(2x) &= 1-2\sin^2 x
\end{align*}
to compute
$\sin(\alpha_k/4$), $\cos(\alpha_k/4$),
$\sin(\alpha_k/2$), $\cos(\alpha_k/2$),
and, finally, $\sin\alpha_k$.
The time complexity to compute  $\sin\alpha_k$ is $O(M(3^k)k^2)$.
The time complexity of Algorithm~\ref{alg:pi} is then $O(M(n)\log^2 n)$.
\end{proof}

If there is an $O(M(n)\log n)$ algorithm to compute $\sin x$ without requiring $\pi$
for any $n$-bit number $x$ with $|x|<U$, a fixed upper bound,
then Algorithm~\ref{alg:pi} runs in $O(M(n)\log n)$.
We aware that the binary splitting algorithms described in \cite{Karatsuba1991, Karatsuba2001}
may be able to compute sine and cosine in $O(M(n)\log n)$.
Unfortunately, we do not have a proof so that we have Conjecture \ref{conj:O(M(n)log n)}.

Note that the AGM sine algorithm \cite{Brent1976agm},
which runs in $O(M(n)\log n)$,
cannot be used here
since it requires $\pi$ as an input.
Note also that the binary splitting algorithm can be used to compute $\pi$ directly \cite{Karatsuba1991}.
However, the time complexity is $O(M(n)\log^2 n)$.

\subsection{A Numerical Example}
The following example has been computed by PARI/GP \cite{PARI} and GMP \cite{GMP}.
We simply have used the sine function provided by PARI/GP.
In the table below,
$\alpha_k$ is the approximated value of $\pi$ in iteration $k$,
$\epsilon_k$ is an upper bound of the error
and $n_k$ is the precision in $\alpha_{k+1}$,
where
\begin{align*}
\alpha_k &= \alpha_{k-1} + \sin\alpha_{k-1}, \qquad\alpha_0=3, \\
\epsilon_k &= (\pi-\alpha_k^3)/6, \\
n_k &= \lfloor-\log_{10}\epsilon_k\rfloor.
\end{align*}
\begin{center}
\begin{tabular}{ | l | l | l | l | l | }
\hline
 $k$ & $\alpha_k$ & $\epsilon_k$ & $n_k$ & $\sin\alpha_k$ \\
\hline
 0 & 3            & $4.73\cdot10^{-4}$  & 3  & \underline{0.141}120008059867222100744802808110 \\
 1 & 3.141        & $3.47\cdot10^{-11}$ & 10 & \underline{0.0005926535}55099468066916718249636 \\
 2 & 3.1415926535 & $1.21\cdot10^{-31}$ & 30 & \underline{0.000000000089793238462643383279}382 \\
\hline
\end{tabular}
\end{center}
We have the following sequence converging to $\pi$,
\begin{align*}
 \alpha_0 &= 3, \\
 \alpha_1 &= 3.141, \\
 \alpha_2 &= 3.1415926535, \\
 \alpha_3 &= 3.141592653589793238462643383279, \\
      \pi &= 3.14159265358979323846264338327950\cdots.
\end{align*}
\section{The Decision Problem}

Let $\alpha$ with $n$ decimal places be a computed value of $\pi$.
How to verify if the digits are correct?
In other words, verify if
\begin{equation}\label{eqn:decision}
10^n\alpha = \left\lfloor10^n\pi\right\rfloor.
\end{equation}
It is interesting to ask if the decision problem,
i.e.\ checking (\ref{eqn:decision}) for a given $\alpha$ in $n$ decimal places,
is easier than the computational problem,
i.e.\ computing $\pi$ in $n$ decimal places.
An algorithm deciding (\ref{eqn:decision}) asymptotically faster than computing $\pi$ has not been discovered.

The self-correcting step in Algorithm~\ref{alg:pi} can be used for verification.
Split
\[
\alpha=\alpha_H + \alpha_L \cdot 10^{-m}
\]
into higher digits and lower digits for some $m>n/3$
such that
\[
\alpha'=\alpha_H + \sin\alpha_H
\]
is expected to have a few more correct digits than $\alpha$.
Check if all the digits in $\alpha$ match $\alpha'$.

In practice,
after $\pi$ is computed in $n$ decimal places by an algorithm,
a different algorithm or the same algorithm with a different set of parameters is used to verify the result.

The $\pi$ result mentioned in the introduction has $\lfloor\pi\cdot10^{13}\rfloor$
decimal digits\footnote{Note that $\lfloor\pi\cdot10^{13}\rfloor=31,415,926,535,897$.}
and 26,090,362,246,629 hexadecimal digits \cite{Yee2019, Iwao2019}.
The computation used the Chudnovsky algorithm.
For verification,
the Bailey–Borwein–Plouffe (BBP) formula \cite{Bailey1997}
and also the Bellard's improved BBP formula \cite{Bellard1997}
were used to compute 48 hexadecimal digits starting at the 26,090,362,246,601st position.
There were 29 hexadecimal digits,
\[
\mbox{from }\quad26,090,362,246,601\mbox{st}\quad\mbox{to}\quad26,090,362,246,629\mbox{th},
\]
agreed in all three results
from Chudnovsky, BBP and Bellard.

In 2010, we computed the two quadrillionth bit of $\pi$ \cite{Sze2011} using Bellard's formula.
Two computations at two different bit positions,
\[
1,999,999,999,999,993\mbox{rd}\quad\mbox{ and }\quad1,999,999,999,999,997\mbox{th},
\]
were executed.
There were 256 bits agreed in both computations.
\section{Convergent Sequences}
We extend Theorem~\ref{thm:cubic} to 
show a family of sequences converging to $\pi\pmod{2\pi}$ in this section.

\begin{lemma}\label{lemma:sequence}
For $k\geq 1$, define
\begin{align*}
a_{k+1} &= a_k + \sin a_k,
\end{align*}
where $0<a_0<\pi$.
Then,
\begin{align*}
\lim_{k\rightarrow\infty} a_k = \pi.
\end{align*}
\end{lemma}
\begin{proof}
We will show
\begin{align}\label{eqn:pi-1<a_k<pi}
\pi-1<a_k<\pi
\mbox{ for some }
k\geq 0.
\end{align}
Then Theorem~\ref{thm:cubic} implies
\begin{align*}
\lim_{k\rightarrow\infty} a_k = \pi.
\end{align*}

Now we show (\ref{eqn:pi-1<a_k<pi}).
If $a_0>\pi-1$, we are done.
Assume $0<a_0\leq\pi-1$.
There exists the least integer $k_0>0$ such that $a_{k_0}>\pi-1$.
If not,
let $U\leq\pi-1$
be the least upper bound of $\{a_k\}$.
Let
\[
\Delta_k = a_{k+1}-a_k = \sin a_k.
\]
Since $\{a_k\}$ is bounded above by $\pi-1$,
we have $a_k>0$ and $\Delta_k>0$ for all $k$.
Since $\{a_k\}$ is increasing with the least upper bound $U$,
\begin{align}\label{eqn:lim-a_k=U}
\lim_{k\rightarrow\infty} a_k=U.
\end{align}
However, (\ref{eqn:lim-a_k=U}) is contradiction
since if $U<\pi/2$, $\Delta_k$ is increasing;
otherwise,
$\Delta_k\geq\sin(\pi - 1)$ for large enough $k$ if $\pi/2\leq U<\pi-1$.
Therefore,
\[
a_{k_0}>\pi-1.
\]
Since $k_0$ is the least integer,
$a_{k_0-1}\leq \pi-1$.
If $a_{k_0-1}=\pi-1$,
we have $\sin(a_{k_0-1})<1$;
otherwise, $a_{k_0-1}<\pi-1$ and $\sin(a_{k_0-1})\leq 1$.
In both cases,
\[
a_{k_0}=a_{k_0-1}+\sin(a_{k_0-1})<\pi.
\]
\end{proof}

For any $a,b,x\in\R$ with $x>0$, define
\[
a \equiv b \pmod{x}
\]
if and only if
\[
a - b = nx
\]
for some integer $n$.
We show a more general theorem below.

\begin{theorem}[Convergent Sequences]\label{thm:sequence}
For any $a_0\in\R$ such that
\begin{align*}
a_0 &\not\equiv 0 \pmod{2\pi}.
\end{align*}
For $k\geq 1$, define
\begin{align*}
a_{k+1} &= a_k + \sin a_k.
\end{align*}
Then,
\begin{align}\label{eqn:lim-a_k}
\lim_{k\rightarrow\infty} a_k
\equiv \pi \pmod{2\pi}.
\end{align}
\end{theorem}
\begin{proof}
If $a_0\equiv\pi\pmod{2\pi}$, it is trivial.
Assume $a_0\not\equiv\pi\pmod{2\pi}$.

Let $n=\lfloor (a_0+\pi)/2\pi\rfloor$ and
\begin{align*}
b_0 &= a_0 - 2n\pi.
\end{align*}
We have $0<|b_0|<\pi$.
For $k\geq 1$, define
\begin{align*}
b_{k+1} &= b_k + \sin b_k.
\end{align*}
It is obvious that, for $k\geq 0$,
\begin{align*}
a_k &= 2n\pi + b_k.
\end{align*}

If $b_0>0$,
Lemma~\ref{lemma:sequence} implies
$\lim_{k\rightarrow\infty} b_k = \pi$.
Then
\begin{align*}
\lim_{k\rightarrow\infty} a_k
= 2n\pi + \lim_{k\rightarrow\infty} b_k
= (2n+1)\pi.
\end{align*}
Suppose $b_0<0$.
Let $c_0 = -b_0$ so that $0<c_0<\pi$.
For $k\geq 1$, define
\begin{align*}
c_{k+1} &= c_k + \sin c_k.
\end{align*}
Lemma~\ref{lemma:sequence} implies
$\lim_{k\rightarrow\infty} c_k =\pi$.
Since $c_k=-b_k$ for all $k$,
\[
\lim_{k\rightarrow\infty} a_k
= 2n\pi + \lim_{k\rightarrow\infty} b_k
=2n\pi -\lim_{k\rightarrow\infty} c_k
=(2n-1)\pi.
\]
\end{proof}

\bibliographystyle{plain}
\bibstyle{plain}
\bibliography{self-correcting-pi}
\end{document}